\documentclass{amsart}

\usepackage{latexsym, amssymb, amsmath, amsthm, url, hyperref, color, graphicx}

\newtheorem{thm}{Theorem}[section]

\newtheorem{lem}[thm]{Lemma}
\newtheorem{cor}[thm]{Corollary}

\begin{document}

\title{On the cop number of generalized Petersen graphs}

\author[T. Ball]{Taylor Ball}
\email{tarball@indiana.edu}

\author[R.W. Bell]{Robert W. Bell}
\email{rbell@math.msu.edu}

\author[J. Guzman]{Jonathan Guzman}
\email{jonrguz@gmail.com}

\author[M. Hanson-Colvin]{Madeleine Hanson-Colvin}
\email{mhansoncol@brynmawr.edu}

\author[N. Schonsheck]{Nikolas Schonsheck}
\email{schonsheck.2@osu.edu}

\begin{abstract}
We show that the cop number of every generalized Petersen graph is 
at most 4.  The strategy is to play a modified game of cops and
robbers on an infinite cyclic covering space where the objective is
to capture the robber or force the robber towards an end
of the infinite graph.  We prove that finite isometric
subtrees are 1-guardable and apply this to determine the exact cop 
number of some families of generalized Petersen graphs.  We also extend
these ideas to prove that the cop number of any connected I-graph is
at most 5.
\end{abstract}

\maketitle

\section{Introduction}

The game of cops and robbers on graphs was introduced by 
Quilliot~\cite{Quilliot_MR703252} and, indpendently, by 
Nowakowski and Winkler~\cite{Nowakowski_Winkler_MR685631}.  
The game is played as follows.  One player, the cop player, is given 
a collection of $k$ pawns called \emph{cops}.  She assigns each cop to 
a vertex of a given undirected graph $G$.  A second player, the robber 
player, is given a single pawn called a \emph{robber}.  He assigns the 
robber to a vertex of $G$.  The players alternate turns, with the cop 
player going first.  On a turn, a player may move any number of her or his 
pawns, by moving each to an adjacent vertex or \emph{passing} by remaining 
at the same vertex.  If, after either player's move, a cop and the robber 
are at the same vertex, the robber is \emph{captured} and the cop player 
wins.  The \emph{cop number}, $c(G)$, of a graph $G$ is the least positive 
integer $k$ such that $k$ cops suffice to capture the robber in a finite
number of moves.  In this game, both players are assumed to have complete 
information about the graph and the positions of the pawns.  
Bonato and Nowakowski~\cite{Bonato_Nowakowski_MR2830217} 
have written a text which introduces and surveys many of the 
foundational papers on the game of cops and robbers on graphs.

The cop number of a graph is computationally expensive to compute.  An 
algorithm described by Bonato and 
Chiniforooshan~\cite{Bonato_Chiniforooshan_MR2648337} will
check whether or not $k$ cops suffice to win on a given 
graph $G$; however, the algorithm runs in $O(n^{3k+3})$ time, where $n$ is 
the order of the graph. 

There are a number of results on bounds for the cop 
number in terms of a graph invariant.  For example, 
Aigner and Fromme~\cite{Aigner_Fromme_MR739593} 
proved that if the minimum degree of a graph $G$ is $\delta$ and if
$G$ has girth least five, then $c(G) \geq \delta$.  
Frankl~\cite{Frankl_MR890640} generalized this as follows: 
for each integer $t \geq 1$, if the minimum degree of a graph $G$ is 
$\delta \geq 2$ and $G$ has girth at least $8t - 3$, then 
$c(G) > (\delta-1)^t$.  


The present article will establish bounds for the cop number
of generalized Petersen graphs.  Let $n$ and $k$ be a positive integers 
such that $n \geq 5$ and $1 \leq k < n/2$.  The 
\emph{generalized Petersen graph}, $GP(n,k)$, is the undirected graph 
having vertex set $A \cup B$, where $A = \{a_1, \dots, a_n\}$ and
$B = \{b_1, \dots, b_n\}$, and having the following edges: 
$(a_i, a_{i+1})$, $(a_i, b_i)$, and $(b_i, b_{i+k})$ for 
each $i = 1, \dots, n$, where indices are to be read modulo $n$.  

In Section~\ref{S:Bound}, we prove the main result of this article:

\begin{thm}
The cop number of every generalized Petersen graph is less than or equal
to four.
\end{thm}

It is immediate that generalized Petersen graphs are 3-regular, and it
is straightforward to check that the generalized Petersen graph $G(n,k)$ 
has girth at least five if and only if $k \neq 1$ and $n \neq 3k, 4k$. 
For these graphs, the bounds of Aigner, Fromme, and Frankl establish that 
the cop number of a generalized Petersen graph is at least three.  

By implementing the algorithm of Bonato and Chiniforooshan, we have 
confirmed that there are generalized Petersen graphs with cop number
greater than 3; and so, by the theorem above, these graphs have cop number
equal to 4.  

We address the problem of determining the exact cop number of generalized 
Petersen graphs in Section~\ref{S:Exact}.  A variety of ad hoc techniques 
are used; however, one technique which may be of more general interest is
the following (see Section~\ref{S:isometric}):

\begin{thm}
If $T$ is a finite isometric subtree of a graph $G$, then $T$ is 
1-guardable.
\end{thm}

In Section~\ref{S:I-graphs}, we generalize our results to $I$-graphs. 
If $n \geq 5$ and $0 < j,k < n/2$, the \emph{$I$-graph} $I(n,j,k)$ has 
vertex set $A \cup B$ and has the following edges:
$(a_i, a_{i+j})$, $(a_i,b_i)$, $(b_i, b_{i+k})$ for each $i = 1, \dots, n$.
Thus, $I$-graphs are like generalized Petersen graphs with two parameters: 
one for the $A$-vertices and one for the $B$-vertices; in particular,
setting $j=1$, we see that $I(n,1,k) = GP(n,k)$.

\begin{thm}
The cop number of every connected $I$-graph is less than or equal to five.
\end{thm}

Acknowledgement:
We are thankful for encouragement and support from the participants in the
2014 summer research experience for undergraduates at Michigan State 
University where this research was conducted.  We are grateful for 
support from the National Security Agency, the National Science 
Foundation (NSF grant \#DMS-1062817), and Michigan State University.

\section{Infinite cyclic coverings of generalized Petersen graphs}

We define an infinite analogue of a generalized Petersen graph for each 
positive integer $k$.  Let $A = \{a_i \in \mid \mathbb{Z}\}$, 
$B = \{b_i \mid i \in \mathbb{Z}\}$. 
The infinite graph $GP(\infty, k)$ has vertex set $A \cup B$
and has the following edges: $(a_i, a_{i+1})$, $(a_i,b_i)$, and 
$(b_i,b_{i+k})$ for each $i \in \mathbb{Z}$. 
There is a graph homomorphism $\pi: GP(\infty,k) \to GP(n,k)$ given by 
reducing the index of each vertex modulo $n$.  This map induces a regular 
covering map of the geometric realizations of these graphs.  More 
precisely, there is a $\mathbb{Z}$--action on $GP(\infty, k)$ defined as 
follows.  Let $\tau$ be a choice of generator of $\mathbb{Z}$, let
$n \geq 5$, and define $\tau.a_{i} = a_{i+n}$ and $\tau.b_{i} = b_{i+n}$.  
Then $\tau$ extends uniquely to an automorphism of $GP(\infty, k)$ and the 
orbit space is isomorphic to $GP(n,k)$.

The discussion below holds for any covering space
$\pi: \widehat{G} \to G$ of graphs.  The reader may prefer to 
concentrate on the special case where $\widehat{G} = GP(\infty,k)$ and 
$G = GP(n,k)$ for some fixed choices of $n$ and $k$.

Suppose that $X$ is a pawn assigned to a vertex $v$ of $G$.
Let $\pi^{-1}(X) = \{X_w \mid w \in \pi^{-1}(v)\}$ be a set of 
pawns in one-to-one correspondence with the set of pre-images of $v$.
We assign the pawn $X_w$ to the vertex $w$ of $\widehat{G}$.  Suppose 
that $X$ moves from 
$v$ to $v' \neq v$ in $G$ and let $e = (v,v')$ be the corresponding 
edge (oriented from $v$ to $v'$).  There is a unique \emph{lifted move} 
for each $X_w \in \pi^{-1}(X)$: the pawn $X_w$ moves from $w$ to $w'$ 
where $w'$ is the unique endpoint of the lift at $w$ of the edge $e$.  
For example, in $GP(n,k)$, the edge $(a_{n-1},a_n = a_{0})$ will lift at 
$a_{n-1}$ to $(a_{n-1}, a_n)$ in $GP(\infty,k)$; and it will lift at 
$a_{2n-1}$ to $(a_{2n-1}, a_{2n})$, etc.  If $X$ passes, then so does 
each pawn in $\pi^{-1}(X)$.  

Thus, each move of a pawn in $G$ defines a unique move for each pawn 
in its pre-image.  Conversely, a move of any one pawn in $\pi^{-1}(X)$ 
defines a move for $X$ by projecting this move via $\pi$ 
from $\widehat{G}$ to $G$.  We refer to this interplay informally as 
the ``lifted game''.

Suppose that $C$ is a cop on $v \in V(G)$ and $C' \in \pi^{-1}(C)$ is
a cop on $w \in \pi^{-1}(v)$.   A move $e = (w,w')$ by $C'$ defines a 
move for every cop in $\pi^{-1}(C)$ as follows: each $C'' \in \pi^{-1}(C)$
plays the move defined by lifting the edge $\pi(e)$ at the vertex 
which $C''$ occupies.  Thus, a move by one
$C' \in \pi^{-1}(C)$ defines a unique consistent move for every cop in
$\pi^{-1}(C)$, where consistency means that each of these moves projects
to the same move for $C$. When the cops in $\pi^{-1}(C)$ move in this way, 
we refer to $\pi^{-1}(C)$ as a \emph{squad} and say that these cops 
\emph{move as a squad} consistent with the moves of a chosen 
\emph{lead cop} $C'$.  

When each pre-image of a cop or a robber plays as a squad, there is no 
difference between the game played on $\widehat{G}$ and the game played 
on $G$.  The purpose of playing the lifted game is to reveal strategies 
which may not be apparent when one studies only the structure of $G$.  
The fundamental observation is that a sequence of moves in $\widehat{G}$
which results in the capture of any lift of the robber by any lift of a 
cop projects to a sequence of moves in $G$ which results in a capture of 
the robber.

We say that the \emph{weak cop number} of a covering of graphs 
$\pi:\widehat{G} \to G$ is less than or equal to $m$ if $m$ squads 
playing on $\widehat{G}$ can capture a single robber or force him 
to move arbitrarily far away from a fixed choice of a base vertex. 
This definition is independent of the base vertex if $\widehat{G}$ is
connected.


If $\widehat{G}$ is a finite sheeted covering of $G$, then the weak
cop number is equal to the cop number of $G$.  But if $G$ is an infinite 
sheeted cover, the weak cop number can be strictly less than the cop 
number.  For example, for each positive integer $k$, the infinite path 
having vertices $\mathbb{Z}$ and edges 
$\{(n, n+1) \mid n \in \mathbb{Z}\}$ covers the $k$-cycle by reducing 
each vertex modulo $k$.  This covering space has weak cop number 1; but,
for each $k \geq 4$, the $k$-cycle has cop number 2.

It is straight-forward to establish that 
the cop number of $G$ is greater than or equal to the weak cop number
of a covering $\pi: \widehat{G} \to G$ and equality holds if and only if 
the squads have a capture strategy for any single robber in $\widehat{G}$.

The notion of a weak cop number was introduced by Chastand, Laviolette, 
and Polat~\cite{Chastand_Laviolette_Polat_MR1781285}.  Lehner~\cite{Lehner},
in a recent preprint, argues in favorite of the following definition
which is similar to the one used here: a graph $G$ is weakly copwin
if a cop can either capture a robber or prevent him from visiting
any vertex infinitely often.  

\begin{thm} \label{T:Weak_cop_bound}
For each generalized Petersen graph $GP(n,k)$, the weak cop number of 
$\pi: GP(\infty,k) \to GP(n,k)$ is 2.  
\end{thm}

\begin{proof}
Let $\widehat{G} = GP(\infty,k)$ and $G = GP(n,k)$.  It is clear that
one squad is not sufficient since it is assumed that $n \geq 5$.  
Choose $a_0$ as the base vertex in $\widehat{G}$.
Assign two cops, $C_1$ 
and $C_2$, to $a_0$ in $G$.  This determines an assignment of two squads 
$S_1 = \pi^{-1}(C_1)$ and $S_2 = \pi^{-1}(C_2)$ to the vertices of 
$\widehat{G}$.  By abuse of notation, let $C_1$ and $C_2$ denote choices
of lead cops for $S_1$ and $S_2$, respectively, with both assigned to 
$a_0$ in $\widehat{G}$ on the first turn.  Let $R$ denote the robber which 
is assigned to some vertex in $\widehat{G}$.
As in the description of the
lifted game, the moves of each cop in a squad $S_i$ is determined by the 
moves of $C_i$.

The initial strategy of the cops is to move in such a way that, after 
finitely many moves, one cop, say $C_1$, occupies a vertex whose index is 
congruent modulo $k$ to the index of the vertex which the robber 
occupies.  Hereafter, we refer to the index of the vertex which a pawn
occupies as the \emph{index} of the pawn.  Since there are only finitely 
many residues modulo $k$, by moving $C_1$ from $a_0$ to $a_1$ to $a_2$, 
etc. and moving $C_2$ from $a_0$ to $a_{-1}$ to $a_{-2}$, etc., this is 
achieved in less than or equal to $k/2$ turns.  The important observation 
is that the robber can only change the residue modulo $k$ of his index by 
at most one: if $R$ moves within the subgraph $A$ induced by the vertex 
set $A$, then his residue changes by one if and only if he does not pass; 
if he moves within the subgraph $B$ induced by the vertex set $B$ or if he 
moves from a vertex of $A$ to a vertex of $B$ or vice-versa or if he 
passes, then his residue does not change at all.  

After possibly relabeling our cops and squads, we have that $C_1$'s index
is congruent modulo $k$ to $R$'s index. Caution is needed since the other 
cops in the squad $S_1$ need not have indices congruent to $R$'s index; 
so, our choice of lead cop is important for the squad $S_1$.

The next stage of the cops' strategy is to move $C_1$ to match parity with 
the robber in the sense that both occupy vertices in the same induced 
subgraph, either both in $A$ or both in $B$.  This can be achieved on the 
turn after $C_1$ has achieved a congruent index.  If $R$'s next move is 
to a vertex in $B$, then $C_1$ moves to the unique vertex of $B$ to which 
the cop is adjacent.  If $R$ moves instead to a vertex of $A$, then $C_1$ 
plays a move (possibly passing) in $A$ which maintains the congruence of
their indices.  (There is only one such move if $k > 1$.)  In either 
case, $C_1$ has maintained a congruent index and now matches parity 
with $R$.  

On subsequent turns, $C_1$ moves so that both congruence and parity are
maintained.  Without loss of generality, we may assume that the index of
$C_1$ is less than the index of $R$.  Whenever $R$ moves within $B$, $C_1$ 
has a choice of two moves; we declare that $C_1$ will always move towards
$R$, that is towards the vertex in $B$ with larger index.  If $R$ passes
in $B$ or moves towards $C_1$, then after $C_1$'s move the distance
between the two pawns has decreased.  

The final stage of the cops' strategy is for $C_2$ to move in $A$ 
towards $R$.  If $C_2$ has a higher index than $R$, then choose a new
lead cop, which we will again call $C_2$, for the squad $S_2$ so that the
index of $C_2$ is less than $R$'s index.   On each turn, $C_2$ moves in
$A$ by increasing his index.

To establish that the weak cop number is two, we prove that $R$'s 
index cannot remain bounded from above.  Whenever $R$ moves
to decrease his index or leave it unchanged, $C_2$ moves closer
or $C_1$ moves closer.  If $R$ succeeds in lowering his index below that
of $C_2$'s index (see Figures~\ref{F:jump1} and~\ref{F:jump2}), 
then to do so he must move in $B$ towards $C_1$.  Each time this happens, 
$C_1$ moves closer by $2k$.  Moreover, the squad $S_2$ at this point can 
simply choose a new lead cop, again called $C_2$, having index lower than
$R$.  Thus, $R$ can only evade $C_2$ in this way finitely many 
times.  Therefore, $R$ can only decrease his index or leave it unchanged 
finitely many times without being captured.
\end{proof}

\begin{figure} 
\includegraphics[scale=.75]{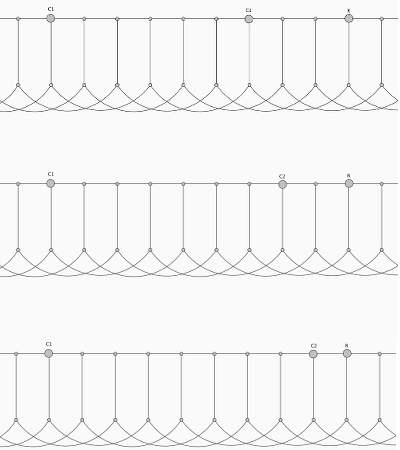}
\caption{As $C_2$ approaches $R$, the robber may have the opportunity
to lower his index below that of $C_2$ (see also Figure~\ref{F:jump2}).}

\label{F:jump1}
\end{figure}

\begin{figure} 
\includegraphics[scale=.75]{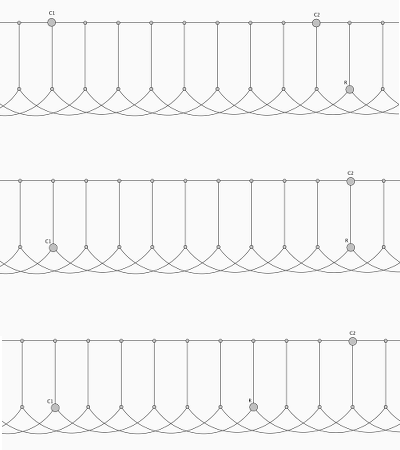}
\caption{To lower his index below that of $C_2$, $R$ allows $C_1$ to
decrease her distance to $R$.}
\label{F:jump2}
\end{figure}

\begin{figure}
\includegraphics[scale=.4]{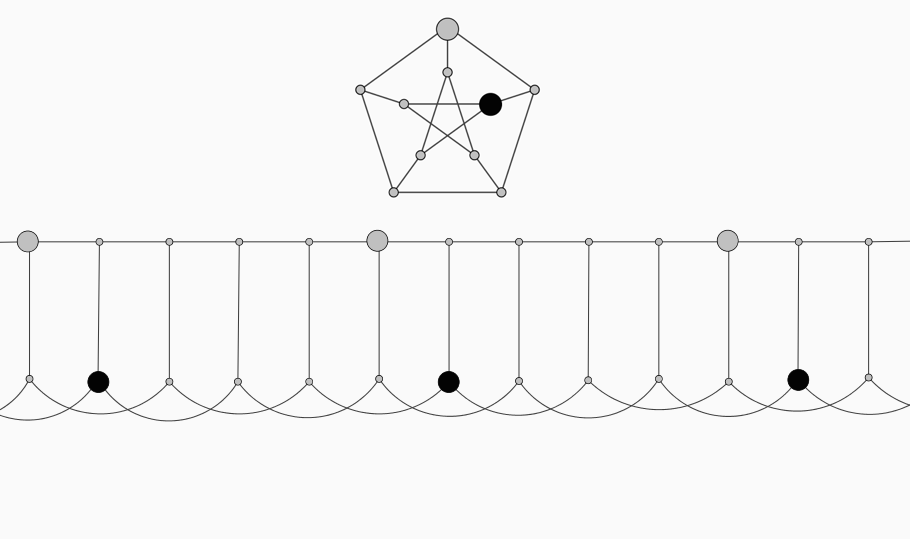}
\caption{A visualization of the lifted game.}
\end{figure}

The following corollary is a refinement of the above proof. It says
that we can, by carefully choosing a new lead cop in $GP(\infty, k)$, 
force the robber to move in a pre-determined direction.

\begin{cor} \label{C:Force_right}
Two squads playing in $GP(\infty, k)$ can capture a single robber or
force his index to increase without bound. 
\end{cor}

\begin{proof}
As in the proof of Theorem~\ref{T:Weak_cop_bound}, two squads $S_1$
and $S_2$ with lead cops $C_1$ and $C_2$, respectively, play against
a single robber $R$.  After at most $k/2 + 1$ turns, $C_1$'s index is 
congruent to $R$'s index and both have the same parity, but it may be the 
case that the index of $C_1$ is greater than $R$'s index.  Following the 
proof of the theorem, it is clear that the robber can be forced to 
decrease his index without bound.  But if we want to force the robber to 
increase his index without bound, we must select a new lead cop for $S_1$ 
which has an index congruent modulo $k$ and which is less than $R$'s 
index.  Since the indices of the cops in $S_1$ are in one-to-one 
correspondence with the elements of the set 
$\{qn + I \mid q \in \mathbb{Z}\}$, where $I$ is the index of $C_1$, we 
can choose a new lead cop corresponding to an index of the form $-mkn + I$ 
for a sufficiently large integer $m$.  Then, following the strategy in 
Theorem~\ref{T:Weak_cop_bound}, we have the desired result.
\end{proof}

\section{Bounding the cop number of the generalized Petersen graph}
\label{S:Bound}

Using the results of the previous section, we prove the main
theorem stated in the introduction.

\begin{thm} \label{T:GP4} The cop number of a generalized Petersen graph 
is less than or equal to four. 
\end{thm}

\begin{proof}
Fix $n$ and $k$ so that $G = GP(n,k)$ is a generalized Petersen graph. 
Let $\widehat{G} = GP(\infty,k) \to GP(n,k)$ be the associated regular 
covering space.  Two pairs of cops play on $G$ to 
capture a robber $R$.  This is achieved by having each cop play the 
projected moves of squads consisting of their pre-images which play the 
lifted game in $\widehat{G}$ against any single robber $\widehat{R}$ in 
the pre-image of $R$.  By Corollary~\ref{C:Force_right}, one pair 
of lifted cops can force $\widehat{R}$ arbitrarily far to the right, 
i.e. force $\widehat{R}$'s index to increase without bound.  A second 
pair of lifted cops can force $\widehat{R}$ arbitrarily far to the left. 
Thus, the four squads will capture $\widehat{R}$.  The covering map $\pi$ 
projects these moves to a capture strategy in $G$.
\end{proof}

There exist generalized Petersen graphs, such as $GP(40,7)$, which have
cop number 4.  That $GP(40,7)$ does not have cop number less than 4 was 
verified with assistance of a a computer.  A summary of our findings and 
more precise results are given in Section~\ref{S:Exact}.

\bigskip


\section{Guarding isometric trees} \label{S:isometric}
The results of this section are of independent interest.  We will will 
use these results to determine the exact value of the cop number of 
several familes of generalized Petersen graphs in Section~\ref{S:Exact}.

The distance, $d_G(u,v)$, between two vertices, $u$ and $v$, in a graph 
$G$ is the length of a shortest path in $G$ joining $u$ to $v$.  
A subgraph $H$ of $G$ is an \emph{isometric subgraph} if for any 
two vertices, $u$ and $v$, in $H$, $d_H(u,v) = d_G(u,v) = d_G(u,v)$

A collection of cops $\{C_i\}$ is said to \emph{guard} a subgraph $H$
of a graph $G$ if these cops occupy vertices of $H$, move only in $H$, and
can end each cop turn so that the following \emph{guarding condition} 
$(GC)$ holds: 
\begin{equation}
\tag{GC}
\forall v \in H, \, \exists C_i, \, d_H(v,C_i) \leq d_G(v,R),
\end{equation}
where $R$ refers to the position of the robber.  If $H$ is guarded by 
$\{C_i\}$ and the robber were to enter $H$, then $(GC)$ implies that the 
robber is immediately captured or captured on the next 
move by one of these cops.

A subgraph $H$ of a graph $G$ is \emph{$k$-guardable} if $k$ cops, 
$C_1, \dots, C_k$, can, after finitely many moves, arrange themselves so that
they guard $H$.  By guarding a subgraph, the cops effectively 
eliminate a portion of the larger graph that the robber can play on.  

Aigner and Fromme~\cite{Aigner_Fromme_MR739593}
showed that finite isometric paths are 1-guardable, and 
used this fact to prove that the cop number of any planar graph is less than 
or equal to three.  We generalize their result to finite isometric trees
below.

\begin{lem} \label{L:same_component}
Suppose that $T$ is an isometric subtree of a graph $G$.  Suppose that
a cop occupies $c \in V(T)$ and the robber occupies $r \in V(G)$.  If
$v_1, v_2 \in V(T)$ belong to different components of $T-\{c\}$ and 
$d(c,v_1) > d(r,v_1)$, then $d(c,v_2) < d(r,v_2)$.
\end{lem}

\begin{proof}
If, contrary to the conclusion, $d(c,v_2) \geq d(r,v_2)$, then 
\[d(v_1, v_2) = d(v_1,c) + d(c,v_2) > d(v_1,r) + d(r,v_2) \geq d(v_1,v_2),\]
which is a contradiction.
\end{proof}

\begin{lem} \label{L:move_closer}
Suppose that $T$ is an isometric subtree of a graph $G$.  Suppose
that a cop occupies $c_0 \in V(T)$ and the robber occupies $r \in V(G)$.
Suppose that $v_0 \in V(T)$ and $d(c,v_0) > d(r,v_0)$.  Let $c_1$ be the 
unique vertex in $T$ adjacent to $c_0$ and closer to $v_0$.  If the cop
moves to $c_1$, then

1. If $v \in T$ belongs to a different component of $T-\{c_0\}$ than
$v_0$, then $d(c_1, v) \leq d(r, v)$.

2. If $v \in T$ belongs to the same component of $T-\{c_0\}$ as $v_0$, then
$d(c_1, v) = d(c_0, v) - 1$.  
\end{lem}

\begin{proof}
Suppose that the cop moves to $c_1$ as described above.  If $v \in V(T)$
is in a different component of $T-\{c_0\}$ than $v_0$, then, by
Lemma~\ref{L:same_component}, $d(c_0, v) < d(r, v)$.  Therefore,
$d(c_1, v) \leq d(r, v)$.  
If $v \in V(T)$ lies in the same component
of $T - \{c_0\}$ as $v_0$, then $d(c_1, v_0) = d(c_0, v_0) - 1$ because
$T$ is an isometric tree.  
\end{proof}

\begin{thm} \label{T:isometric_tree}
If $T$ is a finite isometric subtree of a graph $G$, then $T$ is 
$1$-guardable.
\end{thm}

\begin{proof}
There are two parts to the argument.  First, we must show that a cop
can end her turn so that the guarding condition $(GC)$ holds for $T$.
Second, we must show that after any subsequent move by the robber, she
can move in $T$ so that $(GC)$ still holds after her move.

We may assume that the cop begins on a vertex $c_0 \in V(T)$.  Let 
$U_0$ be the set of vertices of $T$ which belong to the component of 
$T - \{c_0\}$ which contains all vertices of $T$ for which $(GC)$ fails. 
This is well-defined by Lemma~\ref{L:same_component}.  If $U_0 = \emptyset$, 
then the cop can pass and $(GC)$ holds.  Otherwise, the cop moves to the 
unique adjacent vertex $c_1 \in U_0$.  Let $U_1$ be defined analogously.
Lemma~\ref{L:move_closer} implies that $U_1$ is a proper subset of $U_0$.
If $U_1 = \emptyset$, then the first part is complete.  If 
$U_1 \neq \emptyset$ and the robber moves from $r_0$ to $r_1$, then 
Lemma~\ref{L:same_component} implies that the only vertices of $T$ which are 
closer to $r_1$ than to $c_1$ still belong to $U_1$.  The reason is that any 
vertex of $V(T) - U_1$ is farther from $r_0$ than from $c_1$.  Therefore,
this strategy can be continued.  Since $T$ is assumed to be a finite tree,
there is a $k \geq 0$ such that $U_k = \emptyset$.

Suppose that $(GC)$ holds and that the robber is to move.  Let $c_0$ be
the vertex occupied by the cop.  If the robber moves from $r_0$ to $r_1$ so
that $d(r_1,v_1) < d(c_0, v_1)$ for some vertex $v_1 \in V(T)$, then the
cop moves to $c_1$, following the same strategy as above.  Since 
$d(c_0,v_1) \leq d(r_0,v_1)$, we have that $d(c_1, v_1) \leq d(r_1,v_1)$.
Moreover, by Lemma~\ref{L:move_closer}, $(GC)$ holds.
\end{proof}

\section{Determining the exact cop number of $GP(n,k)$} \label{S:Exact}

The graphs below have cop number 2.

\begin{itemize}
\item $GP(6,2)$: choose two antipodal vertices on the \emph{outer rim}, 
i.e. the induced subgraph on $A$.  Up to symmetry there is only one safe 
starting position for $R$.  Then the cops can move to block all moves of 
$R$ with their next move.  
\item $GP(8,2)$: choose two antipodal vertices on the outer rim.  Up
to symmetry, there are three safe starting positions for $R$.  A 
case-by-case analysis establishes that two cops suffice.
\item $GP(n,1)$: place on cop on a vertex of the outer rim and place
the other cop on the adjacent vertex of the inner rim.  Both cops
move in opposite directions around their respective cycles.  
\item $GP(9,3)$ and $GP(12,3)$ have cop number 2; this has been verified
with a computer.
\end{itemize}

The only other candidates for having cop number two are those of the form
$GP(3k,k)$ or $GP(4k,k)$, where $k \geq 2$.  It has been verified with
a computer that $c(GP(12,4)) = 3$ and $c(GP(16,4)) = 3$.  

We have verified with a computer that three cops do not suffice for each 
of the graphs in Figure~\ref{F:cop4}; hence, by Theorem~\ref{T:GP4} these 
graphs have cop number four.

\begin{center}
\begin{figure} \label{F:cop4}
\begin{tabular}{|c|l|}
\hline
$n$& $k$\\
\hline
26& 10\\
27& 6\\
28& 6, 8\\
29& 8, 11, 12\\
31& 7, 9, 12, 13\\
32& 6, 7, 9, 12\\
33& 6, 7, 9, 14\\
34& 6, 10, 13, 14\\
35& 6, 8, 10, 13, 15\\
36& 8, 10, 14, 15\\
37& 6, 7, 8, 10, 11, 14, 16\\
38& 6, 7, 8, 11, 14, 16\\
39& 6, 7, 9, 11, 15, 16, 17\\
40& 6, 7, 9, 11, 12, 15\\
\hline
\end{tabular}
\caption{Above is a complete list of the generalized Petersen graphs $GP(n,k)$,
with $n \leq 40$, which have cop number four.}
\end{figure}
\end{center}

\begin{thm}
The cop number of $GP(n,3)$ is less than or equal to three.
\end{thm}

\begin{proof}
Guard the isometric tree in $GP(\infty,3)$ having vertex set 
$\{a_1,a_2,a_3,b_1,b_2,b_3\}$.  This disconnects $GP(\infty,3)$.  
One cop in $GP(n,3)$ plays the lifted strategy of guarding the
isometric tree in $GP(\infty,3)$.  The other two cops play the
lifted strategy of pushing the robber to the right (or left).
\end{proof}


\section{I-graphs} \label{S:I-graphs}

In this section, we use a similar lifting strategy to bound the cop 
number of connected I-graphs.  Given $n \geq 5$ and $0 < j, k < n/2$, 
the \emph{I-graph}, $I(n,k,j)$, is the 
graph with vertex set $\{a_0, a_1, \ldots, a_{n-1}, b_0, b_1,\ldots,
b_{n-1}\}$ and having edges of the form 
$(a_i, a_{i+j}), \ (a_i,b_i)$ and $(b_i,b_{i+k})$ for each 
$i=1,2,\ldots,n$ with indices read modulo $n$.  I-graphs, thus,
are similar to generalized Petersen graphs except that two parameters 
define the adjacencies. Examples 
include $I(7, 3, 2)$ and $I(8, 2, 3)$, as shown below.

\begin{equation*}
\includegraphics[scale=.15]{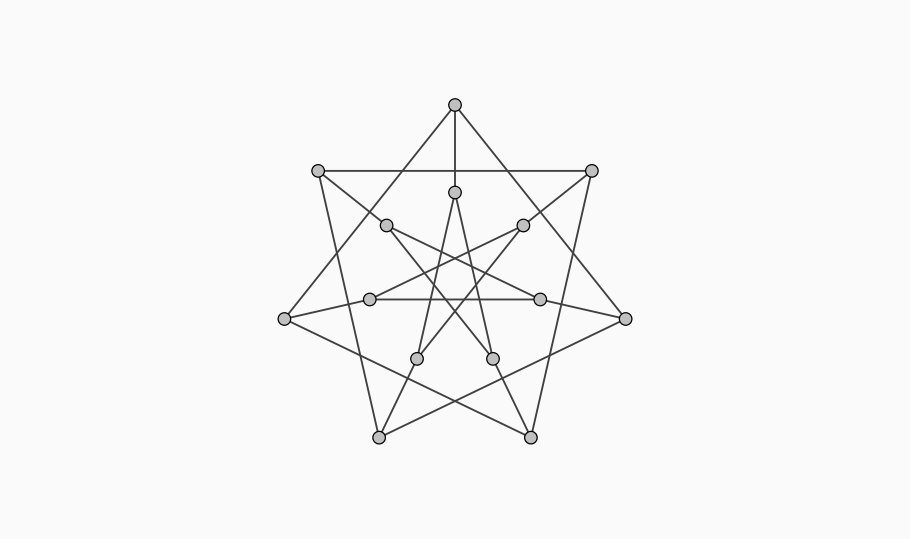} \hspace{1in} 
\includegraphics[scale=.15]{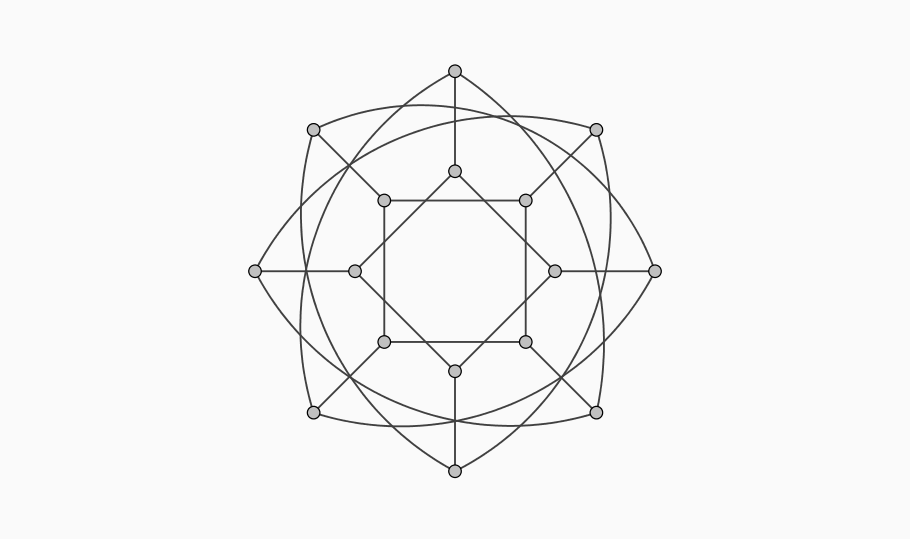}
\end{equation*}

By first examining the game played on the subset of connected I-graphs for 
which $k$ and $j$ are coprime (and bounding the cop number for this family 
of graphs) we are able to bound from above the cop number of all 
connected I-graphs.  To do so, we first define an infinite analogue of an 
I-graph for each $k, j \in \mathbb{N}$ as follows. Given such $k$ and $j$, 
let $A = \{a_i \ | \ i \in \mathbb{Z}\}$ and 
$B = \{b_i \ | \ i \in \mathbb{Z}\}$.  The infinite graph, 
$I(\infty, k, j)$ has vertex set $A \cup B$ and edges: 
$(a_i, a_{i+j}), \ (a_i, b_i), \ (b_i, b_{i+k})$ for each 
$i \in \mathbb{N}$.  Then there is a 
graph homomorphism $\pi:I(\infty,k,j) \to I(n,k,j)$ given by reducing 
indices modulo $n$. In an entirely analogous manner to the lifted game 
for generalized Petersen graphs, we define the lifted game played on 
I-graphs and their corresponding cyclic coverings; each move of a pawn 
in the finite I-graph corresponds to a move for its associated squad in 
the infinite graph, and vice versa. We can then show:

\begin{thm} \label{T:Igraph}
The cop number of a connected $I$-graph $I(n, k, j)$ 
is less than or equal to 5.
\end{thm}

Before proving the main result of the section involving an arbitrary 
connected I-graph, $I(n,k,j)$, we reduce the problem to a simpler one, 
namely with $\gcd(k,j) = 1$. We prove the result for the special case and 
then extend the result to prove our general theorem. 

\begin{thm}
The cop number of a connected I-graph $I(n, k, j)$ with $k$ and $j$ 
coprime is less than or equal to 5.
\end{thm}

\begin{proof}

Let $\widehat{I} = I(\infty, k, j)$ and $I = I(n, k, j)$ with 
$\pi:\widehat{I} \to I$ the associated projection map; let $R$ be the 
robber player on $I$. We describe a strategy in which five squads of cops 
capture one member of $\pi^{-1}(R)$ playing in $\widehat{I}$; then five 
cops can play on $I$, following the projected moves of their corresponding 
squads in $\widehat{I}$ can capture $R$ in a finite number of moves on 
$I$. That is, the strategy of the five squads in $\widehat{I}$ to capture 
a single member of $\pi^{-1}(R)$ will correspond to the five cops' capture 
of $R$ playing on $I$.

Fix one member, $\widehat{R}$ of $\pi^{-1}(R)$ and denote five squads 
$S_1, \ldots S_5$ of cops playing in $\widehat{I}$ with squad leaders 
$C_1, \ldots, C_5$, respectively. Firstly, $C_1$ and $C_2$ move along 
$\pi^{-1}(A)$, with $C_1$ increasing his index by $j$ with each move, 
and $C_2$ decreasing his by $j$ each move. Since $k$ and $j$ are assumed 
to be coprime, one of $C_1$ and $C_2$ will--- in less than 
$\max\{k,j\}/2$ moves--- obtain a congruent index, modulo 
$k$ with $\widehat{R}$. Without loss of generality, suppose $C_1$ 
accomplishes this first. On the next move, $C_1$ can move onto the same 
subgraph as $\widehat{R}$, and will, on subsequent turns, move so as 
to maintain an index congruent, modulo $k$ to that of $\widehat{R}$ and 
to stay on the same subgraph as $\widehat{R}$. Next, $C_2$ and $C_3$ move 
along $\pi^{-1}(B)$, one increasing and the other decreasing index. By 
symmetric reasoning, one of $C_2$ and $C_3$ can match index, modulo $j$ 
with that of $\widehat{R}$. Suppose, without loss of generality, that 
$C_2$ does so first. Then $C_2$ follows an analogous mirroring strategy 
as $C_1$. We can repeat this process twice more, so that (up to relabeling 
squads), $C_1$ and $C_3$ are on vertices which are congruent modulo 
$k$ with that of $\widehat{R}$; and $C_2$ and $C_4$ are on vertices 
congruent modulo $j$ with that of $\widehat{R}$.  Possibly reassigning 
squad leaders, we can also assume that the indices of $C_1$ and $C_2$ 
are strictly less than that of $\widehat{R}$, which is in turn strictly 
less than the indices of $C_3$ and $C_4$.

Once this positioning is obtained, with every move of $\widehat{R}$
within $\pi^{-1}(A)$ or $\pi^{-1}(B)$, 
three of $C_1, C_2, C_3$, and $C_4$ can maintain distance to 
$\widehat{R}$ while maintaining appropriate indices and remaining on the 
same subgraph as $\widehat{R}$. Further, \emph{one} of $C_1, C_2, C_3$, 
and $C_4$ can reduce distance to the robber while maintaining an 
appropriate index and staying on the same subgraph as $\widehat{R}$.  
If the robber does not pass on each turn
or continually switch subgraphs,
he will therefore eventually be captured by one of 
$C_1, C_2, C_3$, or $C_4$. 
Since $\widehat{I}$ is connected,
$C_5$ can move on 
to ensure that the robber does not indefinitely pass
or continually switch subgraphs, 
which forces capture on $\widehat{I}$ and describes a corresponding 
winning strategy for five cops playing on $I$.
\end{proof}

Boben, Pisanski, and \v{Z}itnik proved that the I-graph $I(n,k,j)$ is 
connected if and only if 
$\gcd(n,k,j) = 1$~\cite{Boben_Pisanski_Zitnik_MR2221849}.
Thus, to obtain Theorem~\ref{T:Igraph}, we may simply reason as follows.  
If $I(n, k, j)$ is a connected $I$-graph, then $\gcd{(n, k, j)} = 1$. 
Therefore, in the case for which we have $\gcd{(k, j)} > 1$, we may select 
our five lead cops to be in the same connected component of $\widehat{I}$ 
as our fixed $\widehat{R}$.  Their same strategy of Theorem 2 in the 
lifted game follows and projects down to a capture on the finite graph 
$I(n, k, j)$.

\bibliography{gen_petersen_arxiv}{}
\bibliographystyle{plain}

\end{document}